\documentclass[a4paper,12pt]{article}
\usepackage[utf8]{inputenc}
\usepackage{amsmath}
\usepackage{amssymb}

\usepackage{amsthm}
\usepackage{graphicx}
\graphicspath{ {./graphs/} }

\newtheorem{theo}{Theorem}
\newtheorem{prop}{Proposition}
\newtheorem{lem}{Lemma}

\def\dom{\mathrm{dom}\,}
\def\R{\mathbb{R}}
\title{Hadamard Inverse Function Theorem\\ Proved by Variational Analysis}
\author{Milen Ivanov and Nadia Zlateva}
\date{\em In memoriam to A. L.  Dontchev}

\begin{document}

\maketitle

\begin{abstract}
 We present a proof of Hadamard Inverse Function Theorem by the methods of Variational Analysis, adapting an idea of I. Ekeland and E. S\'er\'e~\cite{ES}.
\end{abstract}

\section{Introduction}
The classical example $(x,y)\to e^x(\cos y, \sin y)$ shows that -- except in dimension one -- the derivative may be everywhere invertible while the function itself is invertible only locally. Probably the historically first sufficient condition  for global invertibility is given by J. S. Hadamard, see \eqref{eq:2} in Theorem~\ref{thm:hadamard} below.

An excellent overview -- both from research and educational perspective -- of this topic is given in \cite{plastock}. Perhaps the easiest to understand -- because of its geometrical nature -- proof involves application of the Mountain Pass Theorem to the function $x\to\|f(x)-y\|$ to ensure the injectivity of $f$. However, Mountain Pass Theorem -- although ``obvious'' -- is hard to verify and may impose additional restrictions of technical nature.

Here we present a new proof based on a recent idea by I. Ekeland and E. S\'er\'e~\cite{ES}. This idea allows obtaining a continuous right inverse to $f$ on any compact, see Proposition~\ref{pro:main}. The necessary key Proposition~\ref{pro:main} is proved by methods of Variational Analysis in the flavour of the  monographs of A. Dontchev~\cite{Asen-book}, A. Dontchev and T. Rockafellar~\cite{doro}, and A. Ioffe~\cite{ioffe}.

\bigskip

We work in a Banach space $(X, \|\cdot\|)$ and denote its closed unit ball  by~$B_X$. Recall that the function
  $$
    f: X\to X
  $$
  is called Fr\'echet differentiable at $x\in X$ if there is a bounded linear operator  $f'(x):X\to X$ such that
  $$
    \lim_{\|h\|\to 0}\frac{f(x+h)-f(x)-f'(x)h}{\|h\|} = 0.
  $$
  The function $f$ is called {\em smooth}, denoted $f\in C^1$, if the function
  $$
    x \to f'(x)
  $$
  is norm-to-norm continuous.

We present a modern proof to the following classical
\begin{theo} \textsc{(Hadamard)}
\label{thm:hadamard}
  Let $f\in C^1$, $f'(x)$ be invertible for all $x$ and satisfying
  \begin{equation}
    \label{eq:2}
    \|[f'(x)]^{-1}\| \le M,\quad\forall x\in X,
  \end{equation}
  for some $M > 0$.

 Then  $f$ is $C^1$ invertible on $X$.

  In other words, there is $g\in C^1$ such that
  $$
     g(f(x)) =f(g(x))=x,\quad\forall x\in X.
  $$
\end{theo}

The work is organised as follows. In the next Section~\ref{sec:prelim} we give the necessary preliminary known facts. In Section~\ref{sec:mainpro} we prove the key Proposition~\ref{pro:main} and in the final Section~\ref{sec:proof} we complete the proof of Hadamard Theorem~\ref{thm:hadamard}.

\section{Preliminaries}
\label{sec:prelim}

We start with recalling the classical (local) Inverse Function Theorem.
\begin{theo}
 \label{thm:inverse-classical}
  Let $f\in C^1$ and let $f'(x_0)$ be invertible. Then there are $\varepsilon,\delta >0$ such that for each $y$ such that
  $$
     \|y - f(x_0)\| <\varepsilon
  $$
  there is unique $x=:g(y)$ such that $\|x-x_0\| < \delta$ and
  $$
    f(x) = y.
  $$
  Moreover, $g\in C^1$ and
  \begin{equation}
   \label{eq:inv-der}
   g'(f(x_0)) = [f'(x_0)]^{-1}.
  \end{equation}
\end{theo}

The following statements are also well-known.
\begin{lem}
   \label{lem:uniform-on-compact}
   Let $f$ be $C^1$. Let $K\subset X$ be compact and let $r>0$. Then
   $$
     f(x+th) = f(x) + tf'(x)h + o(t)\mbox{ uniformly on }x\in K\mbox{ and }h\in rB_X.
   $$
  More precisely, there is $\alpha(t)\to 0$ as $t\to 0$ such that
  $$
    \sup\{\| f(x+th) - f(x) - tf'(x)h\|: x\in K,\ h\in rB_X\} \le \alpha(t)t.
  $$
\end{lem}

\begin{lem}
  \label{lem:cont-1}
  Let $X$ be a Banach space and let $A(x)$ be bounded linear operator for each $x\in X$. Let the function $x\to A(x)$ be norm-to-norm continuous at $x_0$. If $A(x_0)$ is invertible then
  $$
    x\to A^{-1}(x)
  $$
  is continuous at $x_0$.
\end{lem}

Next is a precursor to Ekeland Variational Principle, see \cite[Chapter 5, Section 1]{AE}. Of course, it easily follows from Ekeland Variational Principle itself, see e.g. \cite[Basic Lemma]{ioffe}. See also the comments concerning the ``Basic  Lemma'' on \cite[p. 93]{ioffe}. Here we present a proof based on what is called in these comments ``simple iteration''.
\begin{lem}
  \label{lem-ek-tem}
  Let $X$ be a Banach space and let $\mu:X \to \R^+\cup\{\infty\}$ be lower semicontinuous and such that for some $ r>0$
  $$
    \forall x:\ 0<\mu (x) < \infty\Rightarrow \exists y:\ \mu(y) < \mu(x) - r\|y-x\|.
  $$
  Then for each $x\in \dom \mu$ there is $y\in X$ such that
  $$
     \mu(y) = 0\mbox{ and }r\|y-x\| \le \mu(x).
  $$
\end{lem}

\begin{proof}
  Fix $x_0\in \dom \mu$ such that $\mu(x_0) > 0$.  Let $x_1,x_2,\ldots,x_n$ be already chosen in the following way.

Set
  \begin{equation}
   \label{eq:nu-def}
     \nu_n= \sup \{ \|x-x_n\|:\ \mu(x) < \mu(x_n) - r\|x-x_n\|\}.
  \end{equation}
  We are given that the set in the right hand side is nonempty, so $\nu_n > 0$. Also, since $\mu\ge 0$, we have that $\nu_n \le  \mu(x_n)/r <\infty$.

  Choose a $x_{n+1}$ such that
  \begin{equation}
   \label{eq:x-n+1-def}
     \mu(x_{n+1}) < \mu(x_n) - r\|x_{n+1}-x_n\|
     \mbox{ and }\|x_{n+1}-x_n\| > \nu_n/2.
  \end{equation}
  Note that
  $$
    \|x_{n+1} - x_0\| \le \sum_{i=0}^n \|x_{i+1}-x_i\|\le \sum_{i=0}^n(\mu(x_i)-\mu(x_{i+1}))/r\le\mu(x_0)/r.
  $$
If  if $\mu(x_{n+1})=0$, we are done. If not, we continue by induction.

  If we would end up with an infinite sequence $(x_n)_0^\infty$, then from the above inequality $\sum_{i=0}^\infty \|x_{i+1}-x_i\| \le\mu(x_0)/r$, so $x_n\to \bar x$ as $n\to\infty$ and $\|\bar x- x_0\|\le \mu(x_0)/r$.  From \eqref{eq:x-n+1-def} it follows that $\nu_n\to 0$.

  If $\mu(\bar x) > 0$ then we can find $\bar y$ such that
  \begin{equation}\label{vuh}
  \mu(\bar y) < \mu(\bar x) - r\|\bar y-\bar x\|.
   \end{equation}
   Since $\mu$ is lower semicontinuous, we will have for all $n$ large enough $\mu(\bar y) < \mu(x_n) - r\|\bar y-x_n\|$. Hence, see \eqref{eq:nu-def},  $\nu_n \ge \|\bar y-x_n\| $ for all $n$ large enough. Since $\nu_n\to 0$, we get that $\bar y =\bar x$ which contradicts \eqref{vuh}.

  So,  $\mu(\bar x) = 0$ and we are done.
\end{proof}

\section{Right inverse \`a la Ekeland \& S\'er\'e}
\label{sec:mainpro}
The following is what distinguishes our proof of Hadamard Theorem.
\begin{prop}
 \label{pro:main}
  Let $f\in C^1$,  $f'(x)$ be invertible for all $x$ and let $f$ satisfy~\eqref{eq:2}. Let $K\subset X$ be compact. Then $f$ has a continuous right inverse on $K$, that is, there is a continuous $g:K \to X$ such that
  $$
    f(g(x)) = x,\quad \forall x\in K.
  $$
  Moreover, if $f(0)=0\in K$ then there is a continuous right  inverse of $f$ on $K$ that satisfies
  $$
    g(0) = 0.
  $$
\end{prop}

\begin{proof}
Let $C(K,X)$ be the space of all continuous functions from $K$ to $X$. It is clear that when equipped with the norm
$$
   \|g\|_\infty := \max_{y\in K} \|g(y)\|
$$
it is a Banach space.

Consider the following measure
$$
  \mu: C(K,X) \to \R^+
$$
of how much a given function $g$ differs from a right inverse of $f$:
$$
  \mu(g) := \max_{y\in K} \| f(g(y)) - y\|.
$$
It is clear that $\mu$ is lower semicontinuous. (It is easy to check that it is continuous but we do not need this.)

The claim  is that there exists $g$ such that $\mu(g) = 0$.

In order to check the condition of Lemma~\ref{lem-ek-tem}, fix $\hat g\in C(K,X)$ such that
$$
  \mu(\hat g) > 0.
$$
Set $u:K\to X$ as
$$
  u(y) := y - f(\hat g(y)).
$$
By definition,
$$
   \mu (\hat g) = \| u \|_\infty.
$$
So, $u$ is not identically equal to zero, because $\mu(\hat g)>0$.

Put
$$
  w( y) :=[f'(\hat g(y))]^{-1} u(y), \quad \forall y\in K.
$$
Because $x\to f'(x)$ is continuous, from Lemma~\ref{lem:cont-1} it follows that
$$
  y\to [f'(\hat g(y))]^{-1}
$$
is norm-to-norm continuous, so $w\in C(K,X)$.

Therefore,  for $t > 0$
$$
  g_t := \hat g + tw \in C(K,X).
$$
Note for future reference that form \eqref{eq:2} it follows that $\|w\|_\infty \le M \|u\|_\infty$, that is
\begin{equation}
  \label{eq:norm-w}
  \|w\|_\infty \le M\mu(\hat g).
\end{equation}

Our next aim is to estimate $\mu(g_t)$. By definition
$$
   \mu(g_t) := \max_{y\in K} \| f(g_t(y)) - y\|.
$$
For $y\in K$ define $\varphi _y: \R^+\to \R^+$ by
$$
  \varphi _y(t) := \| f(g_t(y)) - y\|,
$$
hence
\begin{equation}
 \label{eq:mufi}
 \mu(g_t) := \max_{y\in K} \varphi _y(t).
\end{equation}
Because the set $ \hat g(K)$ is compact and the set $w(K)$ is bounded, from Lemma~\ref{lem:uniform-on-compact} it follows that
$$
  \max_{y\in K}\| f(\hat g(y) +tw(y)) -  f(\hat g(y)) - t f'(\hat g(y))w(y)\| = \alpha(t)t,
$$
where $\alpha(t)\to 0$ as $t\to 0$.
But
$$
   f'(\hat g(y))w(y) =  f'(\hat g(y))  [f'(\hat g(y))]^{-1} u(y) = u(y),
$$
so
$$
   \|f(g_t(y)) -  f(\hat g(y)) - tu(y)\|_\infty =\alpha(t)t.
$$
Therefore, for any $y\in K$
\begin{eqnarray*}
 \varphi _y(t) &=& \| f(g_t(y)) - y\| \\
   &\le&  \|f(\hat g(y)) + tu(y)-y\| +  \|f(g_t(y)) -  f(\hat g(y)) - tu(y)\|\\
   &\le& \|(t-1)u(y)\| + \alpha(t)t.
\end{eqnarray*}
Since $\varphi_y(0) = \|u(y)\|$ we have that for small $t$
$$
  \varphi _y(t) \le (1-t) \varphi _y (0) + \alpha(t)t.
$$
Taking a maximum over $y\in K$, see \eqref{eq:mufi}, we get
$$
   \mu(g_t) \le (1-t) \mu(g_0)+ \alpha(t)t,
$$
or, in other words,
\begin{equation}
  \label{eq:diff}
  \mu( \hat g + tw) \le\mu(\hat g) - t \mu(\hat g) + \alpha(t)t.
\end{equation}
Since $\mu(\hat g) > 0$, for some $\delta > 0$ we then have $|\alpha(t)|< \mu(\hat g)/2$ for $ t\in (0,\delta)$. So,
$$
  \mu( \hat g + tw) < \mu(\hat g) - (t/2)\mu(\hat g), \quad \forall t\in (0,\delta).
$$
From \eqref{eq:norm-w}, which is $ \mu(\hat g) \ge (1/M)\|w\|_\infty$, we get
$$
  \mu( \hat g + tw) < \mu(\hat g) - (1/2M)\|tw\|_\infty, \quad \forall t\in (0,\delta),
$$
and we can apply Lemma~\ref{lem-ek-tem} with $r=1/2M$, $x=\hat g$ and $y= \hat g +(\delta/2)w$, to conclude that $\mu$ vanishes somewhere.

If $f(0)=0\in K$ then we can modify the above by considering instead of $C(K,X)$ the Banach space of continuous $g:K\to X$ such that $g(0)=0$. It is clear that in this case $u(0)=w(0) = 0$ and everything else works in the same way.
\end{proof}

\section{Proof of Theorem~\ref{thm:hadamard}}
\label{sec:proof}
\begin{proof}
It is enough to show that $f$ is bijective.

Let $y\in X$ be arbitrary and set $K=\{y\}$. From Proposition~\ref{pro:main} it follows that there is $x=g(y)$ such that $f(x)=y$. So, $f$ is surjective, i.e. $f(X)=X$.

Let $a,b\in X$ be such that $f(a)=f(b)$. By considering instead of $f$ the function
$$
  x \to f(b-x) - f(b)
$$
we can assume without loss of generality that
$$
  b=0\mbox{ and } f(0)=0.
$$
Then
$$
  f(a) = 0.
$$
Set
$$K := f([0,a]).$$
Since $f$ is continuous, $K$ is compact. From Proposition~\ref{pro:main} there is a continuous
$$
  g:K\to X,\mbox{ such that }g(0) = 0\mbox{ and }f(g(y))=y,\quad\forall y\in K.
$$
Consider
$$
  I := \{t\in [0,1]:\ g(f(ta)) = ta\}.
$$
Obviously, $0\in I$, because $g(0) = 0$. Due to the continuity of $g$ and $f$ the set $I$ is closed and, therefore, compact. Let
$$
   \bar t := \max \{ t:t\in I\}.
$$
Assume that $\bar t < 1$. By the local Inverse Function Theorem, see Theorem~\ref{thm:inverse-classical}, there are $\delta,\varepsilon>0$ such that for each $y\in X$ such that
$$
  \|y-f(\bar t a)\| < \varepsilon
$$
there is unique $x\in X$ such that $\|x-\bar ta\| < \delta$ and
$$
  f(x) = y.
$$
From the continuity of $f$  there is $\mu > 0$ such that for all $t\in(\bar t, \bar t +\mu)\subset (0,1)$ we have $\|ta - \bar ta\| < \delta$, $\|f(ta) -f(\bar t a)\| < \varepsilon$. Moreover,  $\|g(f(ta)) - \bar t a\| =\| ta  - \bar t a\| < \delta$, since $g(f(\bar ta)) = \bar t a$.

Then, because of  $f(ta)\in K$ we have that $ f(g(f(ta))) = f(ta)$. From the uniqueness of the solution to $f(\cdot )= f( ta)$ in this neighbourhood we get $g(f(ta)) = ta$ for all $t\in(\bar t, \bar t +\mu)$ which contradicts the definition of $\bar t $.

So, $\bar t =1$ meaning that $g(f(a)) = a$. But $f(a) = 0$. Since $g(0)=0$, it follows that $a=0$.

We have proved that if $f(a)=f(b)$ then $a=b$, so $f$ is injective.
\end{proof}


\begin{thebibliography}{99}

\bibitem{AE}
J.-P. Aubin and I. Ekeland,  \emph{Applied nonlinear analysis}, John Wiley \& Sons, New York, 1984, ISBN: 0-486-45324-3

\bibitem{Asen-book}
Asen L. Dontchev, \emph{Lectures on Variational Analysis}, 
Book Series: Applied Mathematical Sciences, Springer, 2021,  ISBN: 978-3-030-79910-6

\bibitem{doro} A. L. Dontchev and R. T. Rockafellar, \emph{Implicit Functions and Solution Mappings:
A View from Variational Analysis}, Series in Operations Research and Financial Engineering, Springer, 2014, ISBN: 978-1-4939-1037-3


\bibitem{ES}
I. Ekeland and E. S\'er\'e, A local surjection theorem, 2017,
https://project.inria.fr/brenier60/files/2011/12/Brenier.pdf



\bibitem{ioffe}  A. Ioffe, \emph{Variational Analysis of Regular Mappings:
Theory and Applications}, Springer Monographs in Mathematics, 2017, SBN: 978-3-319-64277-2

\bibitem{plastock} R. Plastock, Homeomorogisms between Banach spaces, Trans. Amer. Math. Soc., 200, 1974, 169--183.

\end{thebibliography}
\end{document}